\documentclass[dvipdfmx]{amsart}

\makeatletter
\@addtoreset{equation}{section}

\makeatother
\usepackage{eucal}
\usepackage{amsfonts, amsmath, amssymb, amsthm}
\usepackage{url}
\usepackage{cleveref}
\usepackage[dvips]{graphicx}
\usepackage[dvips]{color}
\usepackage{amscd}
\usepackage[right]{lineno}
\usepackage{cleveref}
\usepackage{enumerate}
\usepackage[all]{xy}
\usepackage{eucal}

\newtheorem{theorem}{Theorem}[section]
\newtheorem{lemma}[theorem]{Lemma}
\newtheorem{proposition}[theorem]{Proposition}

\theoremstyle{definition}
\newtheorem{definition}[theorem]{Definition}

\DeclareMathOperator{\ext}{Ext}

\newcommand{\teich}{\mathcal{T}}

\newcommand{\moduli}{\mathcal{M}}

\newcommand{\mcg}{{\rm Mod}}

\newcommand{\ray}{{\boldsymbol r}}
\newcommand{\busemann}{{\boldsymbol b}}

\newcommand{\mf}{\mathcal{MF}(\Sigma_g)}
\newcommand{\pmf}{\mathcal{PMF}(\Sigma_g)}

\newcommand{\hyperbolic}{\mathbb{H}}

\AddToHook{env/lemma/begin}{\crefalias{theorem}{lemma}}
\AddToHook{env/corollary/begin}{\crefalias{theorem}{corollary}}
\AddToHook{env/proposition/begin}{\crefalias{theorem}{proposition}}
\AddToHook{env/remark/begin}{\crefalias{theorem}{remark}}

\begin{document}
%\frontmatter

\title{Limit sets of mapping class groups}
\author{ Hideki Miyachi}
\author{Ken'ichi Ohshika}

\date{\today}

\address{School of Mathematics and Physics,
College of Science and Engineering,
Kanazawa University,
Kakuma-machi, Kanazawa,
Ishikawa, 920-1192, Japan
}
\email{miyachi@se.kanazawa-u.ac.jp}

\address{Faculty of Science, Gakushuin University, Toshima-ku, Tokyo 171-8588, Japan, and
Max-Planck-Institut für Mathematik, Vivatsgasse 7, 53111 Bonn, Germany
}
\email{ohshika@math.gakushuin.ac.jp}

\thanks{This work is partially supported by JSPS KAKENHI Grant Numbers
25K00909,
23K22396,
20K20519}

\begin{abstract}
We introduce the notions of horospherical limit point and big horospherical limit point in the projective measured foliation space for a subgroup of the mapping class group, which are analogous to  these notions for hyperbolic spaces.
We classify Teichm\"{u}ller geodesic rays directed by projective measured foliations into conical, horospherical, big horospherical limit points, and show that these can be determined by the topological properties of the directing foliations: except for the case of minimal and uniquely ergodic projective measured foliation, whose corresponding ray can be either a conical limit point or a non-conical horospherical limit point.
\end{abstract}

\maketitle

\section{Introduction}
The notion of limit set for a Kleinian group has a long history.
It already appeared implicitly in the work of Poincar\'{e} \cite{Po}. 
The influential works of Ahlfors, Bers, among others,  relate the complements of limit sets to deformation theory of Kleinian groups.
Beardon and Maskit \cite{BeMa} introduced the notion of conical limit point, which they called approximation point.
 According to them, a limit point is said to be conical when it is an accumulation point of an orbit which lies in a fixed neighbourhood of a geodesic ray tending to that point.
The conical limit set coincides with the limit set when the Kleinian group is convex cocompact, and the parabolic fixed points are only non-conical limit sets for geometrically finite groups (Bishop \cite{Bi}).
It was also shown that a geometrically infinite Kleinian group has uncountably many non-conical fixed points (Bishop \cite{Bi2}).

Sullivan \cite{Su} introduced the notion of horospherical limit point of a Kleinian group, which is more general than that of conical limit point.
A limit point is said to be horospherical when an orbit intersects every horoball tangent to the point.
He showed there that horospherical limit points are contained in the \lq conservative part' of the action of the Kleinian group on the sphere at infinity.
Lecuire and Mj \cite{LM} gave a necessary and sufficient condition for a limit point of a geodesic ray to be horospherical but non-conical.
Tukia gave a more general notion which is called big horospherical limit point in \cite{Tu}.
A limit point is said to be a big horospherical limit point when an orbit intersects some horosphere at infinitely many points.
Tukia showed that Sullivan's result on conservative action still holds for the big horospherical limit set.

These notions on limit points still work for higher dimensional hyperbolic spaces, or more generally, Gromov hyperboic spaces and their boundaries at infinity.
Although Teichm\"{u}ller space with the Teichm\"{u}ller metric is not Gromov hyperbolic (Masur-Wolf \cite{MW}), the notions of limit point and limit set for a subgroup of the mapping class group were introduced by McCarthy-Papadopoulos \cite{McPa}, considering the action of the mapping class on the Thurston compactification of Teichm\"{u}ller space.
Also, conical limit points were defined by Kent-Leininger \cite{LeiKen} taking Teichm\"{u}ller geodesic rays as  surrogates of hyperbolic geodesic rays.
We note that although in general a Teichm\"{u}ller geodesic ray may not have a unique landing point at  $\pmf$, which is the sphere at infinity of the Thurston compactification, for the definition of conical limit point, the corresponding Teichm\"{u}ller geodesic rays are in the directions of  uniquely geodesic projective measured foliations, and hence have unique landing points at infinity.
Thus we can regard conical limit points as lying on $\pmf$.

In the present paper, we shall define horospherical limit points and big horospherical limit points for subgroups of the mapping class group.
Since we have also to consider the directions of non-uniquely ergodic projective measured foliations, we cannot use  the sphere at infinity of the Thurston compactification  as the space in which they lie.
Instead, as the boundary at infinity, we consider the \lq ray compactification' of Teichm\"{u}ller space, regarding the ideal  points at infinity of the rays issued at some fixed base point as the boundary points at infinity.
By the theory of Hubbard-Masur \cite{HM}, we can identify the ray compactification with the space of projective measured foliations.
Although the topology of the compactification depends on the base point, and the mapping class group action on $\teich(\Sigma_g)$ does not extend continuously to its action on $\pmf$ in the ray compactification, we shall show that we can define horospherical limit points and big horospherical limit points as lying on $\pmf$ independently of the base point.
Conical limit points can also be regarded as the boundary at infinity $\pmf$ of the ray compactification, since at uniquely ergodic projective measured foliations, the ray compactificaton and the Thurston compactification coincide.

The main theorem of this paper concerns the action of the entire mapping class group.
We shall prove that by looking at the type of  projective measured foliation, we can determine to which kind of limit point it corresponds, except for a uniquely ergodic projective measured foliation, in which case it may correspond to either horospherical, non-conical limit point  or conical limit point.

\begin{theorem}
\label{thm:main1}
Let $\Sigma_g$ be a closed orientable surface of genus $g\ge 2$.
Consider the action of the mapping class group $\mcg(\Sigma_g)$ on the Teichm\"{u}ller space $\teich(\Sigma_g)$.
Then for points $[\lambda]$ in $\pmf$, the following hold.
\begin{enumerate}
\item $[\lambda]$ is a big horospherical limit point, but not a horospherical limit point if and only if $\lambda$ contains a non-singular closed leaf.
\item
If $[\lambda]$ is either non-minimal or  non-uniquely ergodic, and does not contain a non-singular closed  leaf,  then it is a horospherical limit but not a conical limit point.
\item If $[\lambda]$ is a conical limit point, then $\lambda$ is minimal and uniquely ergodic.
\item
If a point $[\lambda]$ is minimal and uniquely ergodic, then there are two possibilities: it is either conical limit point or a non-conical horospherical limit point.
Both cases really occur.
%Let $\Lambda^g_C$, $\Lambda^g_h\subset \pmf$ be the conical limit set and the horospherical limit set of the mapping class group $\mcg(\Sigma_g)$, respectively.
\item The complement of the conical limit set is a null set with respect to the Thurston measure.
\end{enumerate}
\end{theorem}

By this theorem, for the mapping class group action on Teichm\"{u}ller space, we can describe completely the horospherical limit set and the big horospherical limit set in terms of measured foliations deciding directions of rays.

\section{Preliminaries}
Let $\Sigma_{g}$ be a closed oriented surface of genus $g$. Throughout this paper, we assume that $g\ge 2$.

\subsection{Teichm\"uller space}
A marked closed Riemann surface $(M,f)$ of genus $g$ is a pair of a closed Riemann surface of genus $g$, and an orientation preserving diffeomorphism $f\colon \Sigma_{g}\to M$.
We regard marked Riemann surfaces $(M_1,f_1)$ and $(M_2,f_2)$ as equivalent if there is a biholomorphic homeomorphism $h\colon M_1\to M_2$ such that $h\circ f_1$ is homotopic to $f_2$.
%\begin{equation}
%\label{eq:Teichmuller-equivalent}
%\xymatrix{
%\Sigma_{g} \ar[r]^{f_1}\ar[rd]_{f_2} & M_1 \ar[d]^h \\
%& M_2
%}
%\end{equation}
The \emph{Teichm\"uller space} $\teich(\Sigma_g)$ is the set of equivalence classes of marked closed Riemann surfaces of genus $g$.
For $x=(M_1,f_1)$, $y=(M_2,f_2)\in \teich(\Sigma_g)$,
the \emph{Teichm\"uller metric} on $\teich(\Sigma_g)$ is defined by setting the distance $d_T$ between $x$ and $y$ to be
\[
d_T(x,y)=\frac{1}{2}\log \inf_hK(h),
\]
where $h$ ranges over all quasi-conformal maps  from $M_1$ to $M_2$ such that $h\circ f_1$ is homotopic to $f_2$.
The Teichm\"uller metric is  complete and uniquely geodesic, and equipped with this metric, $\teich(\Sigma_g)$ is homeomorphic to $\mathbb{R}^{6g-6}$.

\subsection{Measured foliations}
\label{mf}
A \emph{measured foliation} $F$ on $\Sigma_{g}$ is a codimension-1 foliation on $\Sigma_g$ with isolated singularities of negative indices equipped with a transverse measure which is invariant under holonomies.
% of orders $k_1$, $\cdots$, $k_n$ at $x_1$, $\cdots$, $x_n$ is given by an open cover $\{U_i\}_{i\in I}$ of $\Sigma_{g}$ and non-vanishing $C^\infty$ real-valued $1$-forms $\varphi_i$ on $U_i$ satisfying the following conditions:
%\begin{enumerate}
%\item $\varphi_i=\pm \varphi_j$ on $U_i\cap U_j$;
%\item $k_i\ge -1$ and $k_i\ne 0$. If $k_i=-1$, then $x_i\in \Sigma_g\setminus \Sigma_{g}$; and
%\item at each $x_i$, there is a local chart $(u,v)\colon V\to \mathbb{R}^2$ such that, for $z=u+iv$, one has $\varphi={\rm Re}(z^{k_i/2}dz)$ on $V\cap U_i$ for some branch of $z^{k_i/2}$ on $U_i\cap V$.
%\end{enumerate}
%Such a family $\{(U_i,\varphi_i)\}_{i\in I}$ is called an atlas for $F$.
%For a $C^1$-curve $\gamma\colon [a,b]\to \Sigma_g$ ($\supset \Sigma_{g}$), we define the \emph{$F$-length} $l_F(\gamma)$ of $\gamma$ by
%\[
%l_F(\gamma)=\int_a^b|\varphi_i(\dot{\gamma}(t))|dt,
%\]
%where the integral is understood by subdividing $[a,b]$ so that each subarc lies in a single chart of the atlas.

Two measured foliations are said to be Whitehead equivalent if there is a finite sequence of Whitehead moves, i.e. collapsing an arc on a leaf connecting two singularities, and its opposite operation, which changes one to the other.
Let $\mathcal{S}$ be the set of free homotopy classes of non-contractible simple closed curves on $\Sigma_{g}$.
For $\gamma\in \mathcal{S}$, we define $F(\gamma)$ to be the infimum of $\int_{\gamma_1} dF$ over all simple closed curves $\gamma_1$ in the class $\gamma$.
Two measured foliations $F_1$ and $F_2$ are Whitehead equivalent if and only if $F_1(\alpha)=F_2(\alpha)$ for all $\alpha\in \mathcal{S}$.
(See Fathi-Laundenbach-Po\'{e}naru \cite{FLP}.)

Let $\mf$ be the set of equivalence classes of measured foliations on $\Sigma_{g}$.
The set of all measured foliations has a weak topology with respect to the transverse arcs, and we give $\mf$ its quotient topology.
%The space $\mf$ is topologized as follows: a sequence $(F_n)_{n=1}^\infty$ in $\mf$ converges to $F\in\mf$ if $F_n(\alpha)\to F(\alpha)$ as $n\to \infty$ for all $\alpha\in \mathcal{S}$.
The set $\mathbb{R}_{>0}$ of positive real numbers acts on $\mf$ by multiplying the transverse measures.
%\[
%\mathbb{R}_{>0}\times \mf\ni (t,F)=(t,\{U_i,\varphi_i\}_{i\in I})\mapsto tF=\{U_i,t\varphi_i\}_{i\in I}.
%\]
The quotient space $\pmf$ is called the \emph{space of projective measured foliations} on $\Sigma_{g}$.
Thurston showed that the spaces $\mf$ and $\pmf$ are homeomorphic to $\mathbb{R}^{6g-6}\setminus \{0\}$ and $\mathbb{S}^{6g-7+2m}$, respectively.

For a measured foliation $F$, its support $|F|$ is the underlying foliation with its transverse measure forgotten.
Let $M(|F|)$ be the space of transverse measures supported on $|F|$.
Then there are at most $3g-3$ ergodic measures $m_1, \dots , m_k\ k \leq 3g-3$ in $M(|F|)$ such that every measure $\mu$ in $M(|F|)$ is a linear combination with non-negative coefficients of $m_1, \dots , m_k$.
Therefore, the transverse measure of $F$ is expressed as a sum $\alpha_1 m_1+ \dots + \alpha_k m_k$.
We can regard $|F|$ with the transverse measure $\alpha_i m_i$ a measured foliation (which may not be of full support), and denote it by $G_i$.
Then we can write $F=G_1+\dots + G_k$, which we call the {\em ergodic decomposition} of $F$.
We say that $F$ is {\em uniquely ergodic} when $F$ cannot be decomposed into more than one ergodic measure.

We say that a measured foliation $F$ is {\em minimal} when each of its leaves is dense on $\Sigma_g$, including singular ones.
If  $F$ is not minimal, then there is an incompressible subsurface $S$ of $\Sigma_g$ such that for any simple closed curve $\gamma$ on $S$, we have $F(\gamma)=0$.
In this situation, we can cut $F$ along its singular leaves and isotope the complement into $S$.
There is a unique subsurface which is minimal among such incompressible subsurfaces $S$.
We call this subsurface {\em minimal supporting surface} of $F$.

There is a natural generalisation of geometric intersection number to $\mf \times \mf$, which we define by $$I(F,G)=\int_{\Sigma_g} dFdG,$$ representing  them in their equivalence classes so that they are transverse outside their critical points.
We note that when we write $dF$, we identify $F$ with its transverse measure.

\subsection{Holomorphic quadratic differentials}
Fix a point $x=(M,f)\in \teich(\Sigma_g)$.
Let $\mathcal{Q}_x$ be the space of holomorphic quadratic differentials $q=q(z)dz^2$ on $M$.
% whose $L^1$-norm
%\[
%\|q\|=\int_M|q(z)|dxdy
%\]
%is finite.
The space $\mathcal{Q}_x$ is a complex Banach space with the $L^1$-norm defined by $\|q\|=\int_M|q(z)|dxdy$ for $q \in \mathcal Q_x$.
%A quadratic differential $q\in \mathcal{Q}_x$ may have poles of order $1$ at marked points.

For $q=q(z)dz^2\in\mathcal{Q}_x$, take a covering $\{U_i\}_{i\in I}$ of $M\setminus {\rm Crit}(q)$ such that each $U_i$ is simply connected, where ${\rm Crit}(q)$ denotes the set of zeros of $q$.
On each $U_i$, we choose a branch of $\sqrt{q}=\sqrt{q(z)}dz$ and define
\[
\varphi_i={\rm Re}(\sqrt{q(z)}dz).
\]
Together with the standard local models at zeros  of $q$, the family
\[
F_q=\{(f^{-1}(U_i),f^*\varphi_i)\}_{i\in I}
\]
defines a measured foliation on $\Sigma_{g}$.

\begin{theorem}[Hubbard--Masur {\cite{HM}}, Gardiner {\cite{Gard1}}]
    \label{thm:Hubbard-Masur}
    The correspondence
    \[
    \mathcal{Q}_x\ni q\mapsto F_q\in \mf    \]
    is a homeomorphism.
\end{theorem}

%Let $x=(M,f)\in \teich(\Sigma_g)$.
For $F\in \mf$, the quadratic differential corresponding to $F$ under the inverse of the map above is denoted by $q_{F,x}$.
% in \Cref{thm:Hubbard-Masur} is called the \emph{Hubbard--Masur differential} for $F$ on $x$.
%Conversely for a measured foliation $F\in \mf$, we use the symbol $q

%For any $\alpha\in \mathcal{S}$, there is a unique holomorphic quadratic differential $q_{\alpha,x}$ on $M$ such that almost all vertical trajectories are closed and homotopic to curves in $f(\alpha)$.
%Then
%\begin{equation}
%\label{eq:geometric_intersection_scc}
%F_{q_{\alpha,x}}(\gamma)=i(\alpha,\gamma)=\min\{\#(\alpha_1\cap \gamma_1)\mid \alpha_1\in \alpha,\gamma_1\in \gamma\}
%\end{equation}
%for all $\gamma\in \mathcal{S}$.
%The differential $q_{\alpha,x}$ is called the \emph{Jenkins--Strebel differential} for $\alpha$ on $x$.
%After identifying $\alpha\in \mathcal{S}$ with the measured foliation $F_{q_{\alpha,x}}$, we regard $\mathcal{S}$ as a subset of $\mf$.
%For $\alpha\in \mathcal{S}$ and $t>0$, we define $t\alpha$ to be the measured foliation $tF_{q_{\alpha,x}}$.
%This definition is independent of the choice of $x$.
%Set $\mathcal{WS}=\{t\alpha\in \mf\mid \alpha\in \mathcal{S},t>0\}$.
%By definition,
%\[
%t\alpha(\gamma)=t\,i(\alpha,\gamma).
%\]
%It is known that $\mathcal{WS}$ is dense in $\mf$ (cf. \cite{DH}).
%We define the geometric intersection number between $t\alpha$, $s\beta\in \mathcal{WS}\subset \mf$ by
%\[
%i(t\alpha,s\beta)=ts\,i(\alpha,\beta),
%\]
%where $i(\alpha,\beta)$ on the right-hand side is the original intersection number defined in \eqref{eq:geometric_intersection_scc}.
%Thurston showed that the intersection number on $\mathcal{WS}$ extends continuously to $\mf$.

\subsection{Teichmüller geodesic rays}
Let $x=(M,f)$ be a point in $\teich(\Sigma_g)$.
For a holomorphic quadratic differential $q$ on $M$, its associated Beltrami-Teichm\"{u}ler differentials is defined to be $\mu_t = t\bar q/|q|$ for $t=[0,1)$ .
Let $h_t \colon M \to N$ be a solution of the Beltrami equation $\frac{\partial f}{\partial \bar z}=\mu_t(z)$, for a Riemann surface $N$ homeomorphic to $\Sigma_g$.
If we set $x_t=(N, h_t \circ f)$, then we have $\displaystyle d_T(x, x_t)=\frac{1}{2}\log\frac{1+t}{1-t}$.
Using the measured foliation $F$ with $q_{F,x}=q$, we write the Teichmüler geodesic ray which the $x_t$ constitute by $\ray^F_x(t)$, and call it the Teichm\"{u}ller geodesic ray (issued at $x$) in the direction of $F$.
We note that $\ray^F_x$ depends only on the projective class $[F]$ by definition.

We denote the moduli space of $\Sigma_g$, the quotient space of $\teich(\Sigma_g)$ obtained by forgetting markings, by $\moduli(\Sigma_g)$.
Let $\varpi\colon \teich(\Sigma_g) \to \moduli(\Sigma_g)$ be the projection. 
Teichm\"uller distance on $\moduli(\Sigma_g)$ induces a metric on $\moduli(\Sigma_g)$, which we denote by $d_\moduli$.
To be more concrete, we define the metric $d_\moduli$ by
\begin{equation}
\label{eq:distance_modulispace}
d_{\moduli}(p,q)=\inf_{\varpi(x)=p, \varpi(y)=q}d_T(x,y).
\end{equation}
A Teichm\"{u}ller geodesic ray $\ray^F_x(t)$ is said to be {\em divergent} if $\varpi \circ \ray^F_x(t)$ leaves every compact subset of $\moduli(\Sigma_g)$.
% Let $M$ be a complete hyperbolc surface of genus $g$ with $m$ cusps.
% A \emph{geodesic lamination} is a closed subset of $M$ that is a disjoint union of complete geodesics on $M$.
% A \emph{transverse measure} to a geodesic lamination $L$ is an assignment of a Radon measures on each transverse arc $k$ to $L$ such that every homotopy sending $k$ to another transverse arc $k'$ while respecting $L$ sends the measure defined on $k$ to the measure defined on $k'$.
% A transerse measure of $L$ is said to be \emph{full supported} if any transverse arc $k$, the support of the measure assigned to $k$ is equal to $k\cap L$.
% A \emph{measured lamination} $\mu$ consists of a geodesic lamination $L_\mu$, called the \emph{support} of $\mu$, together with a full support transverse measure for $L_\mu$.
% For example, a closed geodesic homotopic to a simple closed curve on $M$ is a geodesic lamination.
% For a simple closed geodesic $\alpha$, we associate the transverse measure defined as an assignment the Dirac measure on $\alpha\cap k$ to each transverse arc $k$ to $\alpha$. Then, we think of any simple closed geodesic on $M$ as a measured lamination on $M$.

% Let $\ml=\ml(M)$ be the set of measured lamination with compact support. For $\mu$ 

\subsection{Extremal length}
We define the \emph{extremal length} of $F\in \mf$ on $x=(M,f)\in \teich(\Sigma_g)$ by
\[
\ext_{x}(F)=\int_M|q_{F,x}(z)|dxdy.
\]
The extremal length is continuous on $\teich(\Sigma_g)\times \mf$ and satisfies
\begin{equation}
\label{extreme}
e^{-2d_T(x,y)}\le \frac{\ext_y(F)}{\ext_x(F)}\le e^{2d_T(x,y)}
\end{equation}
for all $x,y\in\teich(\Sigma_g)$.
Kerckhoff showed in \cite{Ker} the following equality:
$$d_T(x,y)=\frac{1}{2}\sup_{F \in \mf}\log \frac{\ext_y(F)}{\ext_x(F)}.$$

%\subsection{Thurston compactification}
%By the Koebe uniformization theorem, any closed Riemann surface of genus $g\ge 2$ admits a unique complete hyperbolic structure of finite area.
%
%For $x=(M,f)\in \teich_{g}$ and $\alpha\in \mathcal{S}$, we denote by $\ell_x(\alpha)$ the length of the simple closed geodesic on $M$ homotopic to a curve in $f(\alpha)$.
%We topologize $\teich_{g}\cup \pmf$ as follows.
%A sequence $(x_n)_{n=1}^\infty$ in $\teich_{g}$ converges to $[F]\in \pmf$ if there is a sequence $(t_n)_{n=1}^\infty$ of positive numbers such that
%\[
%\lim_{n\to \infty}t_n\ell_{x_n}(\alpha)=F(\alpha)
%\]
%for all $\alpha\in \mathcal{S}$.
%Under this topology, Thurston showed that $\teich_{g}\cup \pmf$ is homeomorphic to a closed unit ball of dimension $6g-6+2m$, with $\teich_{g}$ corresponding to the interior of the ball.
%This compactification of $\teich_{g}$ is called the \emph{Thurston compactification} of $\teich_{g}$.

\subsection{Mapping class group}
The \emph{mapping class group} $\mcg(\Sigma_{g})$ is the group of isotopy classes of orientation-preserving self-diffeomorphisms of $\Sigma_{g}$.
For a diffeomorphism $\omega \colon \Sigma_g \to \Sigma_g$, we denote its mapping class by $[\omega]$.
For $[\omega]\in \mcg(\Sigma_{g})$, we define the action of $[\omega]$ on $\teich(\Sigma_g)$ by
$$
[\omega]_*(x)=(M,f\circ \omega^{-1})\quad (x=(M,f)\in \teich(\Sigma_g)).$$
% \\
%[\omega]_*(F)(\alpha)&=F(\omega^{-1}(\alpha)) \quad (F\in \mf,\ \alpha\in \mathcal{S}).
%\end{align*}
%The action of the mapping class group on $\mf$ induces a homeomorphic action on the Thurston compactification $\teich(\Sigma_g)\cup \pmf$.
We note that we have $\moduli_{g}=\teich_(\Sigma_g)/\mcg(\Sigma_{g})$.
The mapping class group $\mcg(\Sigma_g)$ also acts on $\mf$,  taking foliation by the map and pushing forward the transverse mesure.

\section{Classification of limit points}
We now define  conical limit points, horospherical limit points, and big horospherical limit points for subgroups of the mapping class group.
As mentioned in Introduction, we regard these limit points as lying in $\pmf$, which is identified with the set of Teichm\"{u}ller geodesic rays issued at a fixed point $x_0 \in \teich(\Sigma_g)$.

%\subsection{Limit points}
%We start by defining limit points.
%Let $G$ be a subgroup of $\mcg(\Sigma_g)$, and fix a point $x \in \teich(\Sigma_g)$.
%We call a projective measured foliation $[F] \in \pmf$ a limit point if there exists a sequence of distinct points $(g_i x)$ such that the Teichm\"{u}ller geodesic connecting $x$ to $g_i x$ converges the ray $\ray^x_F$ uniformly on any bounded interval.
%

\subsection{Conical limit points}
We first recall how conical limit points were defined for Kleinain groups.
For a Kleinian group $\Gamma$, a point $z$ in the sphere at infinity $S^2_\infty$ is said to be a conical limit point of $\Gamma$ if there are a geodesic ray $r$ in $\hyperbolic^3$ tending to $z$, a positive number $\epsilon $, and $x \in \hyperbolic^3$ such that the $\epsilon$-neighbourhood of $r$ contains infinitely many distinct points in $\Gamma x$.

This notion of conical limit point was adapted to the case of the action of a subgroup of $\mcg(\Sigma_g)$ on the Teichm\"{u}ller space $\teich(\Sigma_g)$ by  Leininger and Kent in \cite{LeiKen}.
We here define conical limit points as follows.
We do not assume that conicial limit points are limit points, which, by definition, lie on the Thurston boundary.
Still, as we shall see, our definition is essentially the same as that of Leininger-Kent.

\begin{definition}
Let $G$ be a subgroup of $\mcg(\Sigma_g)$.
We call $[F]\in \pmf$ a \emph{conical limit point} of $G$ if, for any Teichm\"uller geodesic ray $\ray^F_{x_0}$ in the direction of $[F]$, there is a positive number $R$ such that some $G$-orbit intersects the $R$-neighbourhood of $\ray^F_{x_0}$ in an infinite set.
We denote by $\Lambda_C(G)$ the set of conical limit points of $G$.
\end{definition}

\begin{lemma}
\label{Masur}
If $[F]$ is a conical limit point of $G$, then $F$ is minimal and uniquely ergodic.
\end{lemma}
\begin{proof}
By Masur's criterion \cite{Mas},  if $F$ is minimal but not uniquely ergodic, then the Teichm\"{u}ller geodesic $\ray^F_{x_0}$ is divergent, and hence cannot be a conical limit point.

Suppose that $F$ is not minimal.
Then, as explained in \Cref{mf}, $F$ contains a possibly singular compact leaf $l$, and
a boundary component $C$ of the minimal supporting surface of $F$ is homotopic into $l$.
Let $q(t)$ be the unit-area quadratic differential corresponding to the positive tangent vector of $\ray^F_{x_0}(t)$.
Since the Euclidean length of any compact leaf with respect to the flat structure determined by $\sqrt{q(t)}$ goes to $0$, the Euclidean length of $C$ also goes to $0$.
This means that $\ray^F_{x_0}(t)$ is divergent, and hence $F$ cannot be a conical limit point.
\end{proof}

% To check the equivalence between their definition and ours, we prove the following proposition.
%
 Our definition of conical limit points may look different from that of Leininger-Kent at first glance, for we do not assume $[F]$ to be a limit point in the sense of McCarthy-Papadopoulos \cite{McPa}.
 Still, as we shall see below, if $[F]$ is a conical limit point in our definition, it can be regarded as a point in the Thurston boundary and  is a limit point, and hence it is a conical limit point in the sense of Leininger-Kent.
%In their definition, conical limit points of $G$ are assumed to be contained in $\Lambda(G)$ whereas they are not in ours.
%In fact, we can show that conical limit points of our definition are also automatically contained in $\Lambda(G)$ as follows.

\begin{lemma}
For any conical limit point $[F] \in \Lambda_C(G)$, any orbit of $G$ converges to $[F]$ in the Thurston compactification of $\teich(S)$.
\end{lemma}
\begin{proof}
Fix a point $x$ in $\teich(\Sigma_g)$, and let $([\omega_n])$ be a sequence of distinct elements in $G$ and  such that every $[\omega_n]_*(x)$ is contained in the $R$-neighbourhood of a Teichm\"uller geodesic ray $\ray^F_{x_0}$ with direction $[F]\in \Lambda_C(G)$.
Take a point $y_n$ on $\ray^F_{x_0}$ which lies within the distance $R$ from $[\omega_n]_*(x)$.
Let $\gamma_n$ be the Teichm\"{u}ller geodesic connecting $x$ with $[\omega_n]_*(x)$, and $l_n$ its length.
Passing to a subsequence, we can assume that $l_n$ goes to $\infty$.
Then $d_T(x, y_n) \geq d_T(x, [\omega_n]_*(x))-R=l_n-R$.

We denote by $[F_n]$ the projective measured foliation directing $\gamma_n$.
We can assume that  $F$ and $F_n$ have extremal lengths $1$ at $x$ without changing their projective classes.
Then by Minsky's inequality \cite{Min},
\begin{equation}
\begin{split}
I(F_n,F)^2&\le \ext_{y_n}(F)\ext_{y_n}(F_n)\\
&\le e^{2R-2l_n}\ext_{x}(F)e^{2R}\ext_{[\omega_n]_*(x)}(F_n)\\
&\le e^{2R-2l_n}\ext_{x}(F)e^{2R-2l_n}\ext_{x}(F_n)\\
&=e^{4R-4l_n}\to 0.
\end{split}
\end{equation}
Since we  $F$ is minimal and uniquely ergodic by \Cref{Masur}, we have $[F]=\lim_{n\to\infty}[F_n]$.
By Masur's theorem \cite{Mas2}, this implies that $([\omega_n]_*(x))$ converges to $[F]$ in the Thurston compactification.
\end{proof}
Thus we have shown that $[F]$ is a limit point of $G$ in the sense of McCarthy-Papadopoulos.
%
%
%Take a sequence $([\omega_n])$ consisting of distinct elements in $G$ and $x_1\in \teich_{g}$ such that each $(\omega_n)_*(x_1)$ is contained in an $R$-neighbourhood of a Teichm\"uller geodesic ray $\ray^F_{x_2}$ with direction $[F]\in \Lambda_C(G)$.
%By passing to a subsequence if necessary, we may assume that $[\omega_n]_*(x_1)$ converges to some point $p$ in the Gardiner--Masur compactification as $n\to \infty$ (cf.\cite{GaMa}).
%Take $t_n>0$ such that $d_T([\omega_n]_*(x_1),\ray^F_{x_2}(t_n))\le R$.
%We may assume that $t_n\to \infty$ as $n\to \infty$.
%Then
%\begin{align*}
%e^{-2d_T(x_2,[\omega_n]_*(x_1))}\ext_{[\omega_n]_*(x_1)}(F)
%&\le e^{2R}e^{-2d_T(x_2,\ray^F_{x_2}(t_n))}\ext_{\ray^F_{x_2}(t_n)}(F) \\
%&=e^{2(R-t_n)}\to 0
%\end{align*}
%as $n\to \infty$.
%By \cite[Theorem 3]{Miy2}, the limit point $p$ coincides with $[F]$.
%Since $[F]$ is uniquely ergodic, $[\omega_n]_*(x_1)$ converges to $[F]$ in the Thurston compactification (cf. \cite[Corollary 1]{Miy2}).
%Therefore, by \cite[Proposition 8.1]{McPa}, there is $[F']\in \Lambda(G)$ such that $i(F,F')=0$.
%Since $[F]$ is uniquely ergodic, $[F]=[F']\in \Lambda(G)$.

\subsection{Horospherical limit points}
For Kleinian groups, the  notions of horospherical limit point and  horospherical limit set were introduced by Sullivan \cite{Su}.
Given a Kleinian group $\Gamma$, a point $z \in S^2_\infty$ is said to be a horospherical limit point of $\Gamma$ if there is a point $x \in \hyperbolic^3$ such that any horoball centred at $z$ has infinitely many points in $\Gamma x$.
Tukia considered a generalisation of the horospherical limit set, which he called big horospherical limit set.
A point $z\in S^2_\infty$ is said to be a big horospherical limit point of $\Gamma$ is there is some horoball centred at $z$ containing infinitely many points on $\Gamma x$. 

We shall adapt these notions to subgroups of $\mcg(\Sigma_g)$ acting on $\teich(\Sigma_g)$.
Recall that for a geodesic ray $r \colon [0,\infty) \to (X,d)$ in a length space $(X,d)$, the Busemann function is defined to be $\busemann_r(x)=\lim_{t\to\infty} (d(r(t),x)-t)$.
In the case of hyperbolic space $\hyperbolic^3$, for a point $z \in S^2_\infty$, a horoball centred at $z$ coincides with the set of points $\{x \mid \busemann_r(x)= K\}$ for a geodesic ray tending to $z$ and a constant $K$.
The same holds for the hyperbolic space of any dimension, and more generally, for a Gromov hyperbolic space, its horoballs are defined to be level sets of its Busemann function. 
Therefore using the Busemann function $\busemann_r$, a horospherical limit point  for a discontinuous group $\Gamma$ acting on a Gromov hyperbolic space  is defined to be a point at infinity $z$ such that for any $\epsilon >0$, the orbit $\Gamma x$ contains infinitely many points whose Busemann function is less than $\epsilon$.

Although Teichm\"{u}ller space with Teichm\"{u}ller metric is not Gromov hyperbolic, we can define its horospherical limit in a similar way.
%, using the level sets of the extremal length of a measured foliation whose projective class corresponds to a point in consideration on the boundary at infinity provided that the measured foliation is uniquely ergodic.
The viewpoint regarding level surfaces of extremal lengths as horospheres was already mentioned in Gardiner-Masur \cite{GaMa}, but  was justified using Busemann's function in Su-Tan \cite{ST} based on the work of Walsh \cite{Walsh}.
Let us first recall the definition of Busemann's function for Teichm\"{u}ller geodesic rays.

\begin{definition}
For a projective foliation $[F]\in \pmf$ and a basepoint $x_0\in \teich(\Sigma_g)$ we define Busumann's function $\busemann_F^{x_0}$ by
\begin{equation}
\label{eq:Busemann}
\busemann_F^{x_0}(x)=
\lim_{t\to \infty}(d_T(x,\ray^{F}_{x_0}(t))-t)
\end{equation}
for $x\in \teich(\Sigma_g)$.
It is well known that the limit is well defined (including $-\infty$) and that the convergence is locally uniform.
\end{definition}

For Teichm\"{u}ller space with the Teichm\"{u}ller metric and a subgroup $G$ of the mapping class group $\mcg(\Sigma_g)$, we define the horospherical limit points of $G$ as follows.
Recall that in the case of (Gromov) hyperbolic space, the horospheres tangent to a point $z$ on the sphere at infinity are the level sets of Busemann's function at $z$, and therefore, a point at infinity is a horospherical limit point if and only if the associated Busemann's function  tends to $-\infty$.
This shows that the following definition is a natural generalisation of horospherical limit points for (Gromov) hyperbolic spaces to our setting.

\begin{definition}
Fix a point $x_0\in \teich(\Sigma_g)$.
We say that $[F]\in \pmf$ is a  \emph{horospherical limit point} of $G$ if there is a sequence $([\omega_n])_{n=1}^\infty$ consisting of distinct elements in $G$ such that
\[
\busemann_F^{x_0}([\omega_n](x))\to -\infty
\]
as $n\to \infty$ for some (and hence any) $x \in \teich(\Sigma_g)$.
We denote by $\Lambda_h(G)$ the set of horospherical limit points of $G$.

We say that $[F]\in \pmf$ is a {\em big horospherical limit point} of $G$ if there is a sequence  $([\omega_n])_{n=1}^\infty$ consisting of distinct elements in $G$ such that $(\busemann_F^{x_0}([\omega_n](x)))$ is bounded from above.
\end{definition}

Although the definition appears to depend on the basepoint $x_0$, in fact, as we shall see next, it is independent of the choice of the basepoint.

\begin{lemma}
Let $x_0$ and $x_1$ be two points in $\teich(\Sigma_g)$.
Then there exists $D_0$ such that $|\busemann_F^{x_0}(x)-\busemann_F^{x_1}(x)|
\le d_T(x_0,x_1)+D_0$ for any $x \in \teich(\Sigma_g)$.
\end{lemma}

\begin{proof}
By Ivanov \cite[Theorem 3.2]{Ivanov}, there is $D_0$ depending only on $F$, $x_0$ and $x_1$ such that
\[
d_T(\ray^F_{x_0}(t),\ray^F_{x_1}(t))\le D_0
\]
for all $t\ge 0$. 
Therefore, for all $x\in \teich(\Sigma_g)$,
\begin{equation}
\label{eq:Busemann-inequality2}
|(d_T(x,\ray^F_{x_0}(t))-t)-
(d_T(x,\ray^F_{x_1}(t))-t)|
\le d_T(\ray^F_{x_0}(t),\ray^F_{x_1}(t))\le D_0,
\end{equation}
which implies our claim by the definition of Busemann's function.
\end{proof}

Since the Busemann function tends to $-\infty$ along the Teichm\"uller ray in the direction of $F$, we have $\Lambda_C(G)\subset \Lambda_h(G)$.

\begin{proposition}
\label{prop:horospherical_limit_points2}
Let $G$ be a subgroup of $\mcg(\Sigma_{g})$, and $x_0=(M_0,f_0)$ a point in $\teich(\Sigma_{g})$.
For a minimal and uniquely ergodic projective measured foliation $[F]\in \pmf$, the following conditions are equivalent.
\begin{itemize}
    \item[(a)] The projective measured foliation $[F]$ is a horospherical limit point.
    \item[(b)] There exist sequences $(t_n)$ and $(D_n)$ in $[0,\infty)$ with $t_n\nearrow \infty$ and $t_n-D_n\to \infty$, a sequence $([\omega_n])$ in $G$, and $e^{2D_n}$-quasi-conformal maps $h_n\colon M_0\to M_{t_n}$ such that $f_{t_n}$ is homotopic to $h_n\circ f_0\circ \omega_n^{-1}$ for all $n$, where $(M_{t_n},f_{t_n})$ denotes $\ray^{F}_{x_0}(t_n)$.
\end{itemize}
\end{proposition}

\begin{proof}
Suppose first that (a) holds. 
%By Walsh \cite{Walsh}, $\busemann_F^{x_0}(x)=\frac{1}{2}(\log \ext_x(F)-\log\ext_{x_0}(F))$ since $F$ is assumed to be uniquely ergodic.
%Therefore, there is a sequence $(\omega_n)$ in $G$ such that
%\begin{equation}
%\label{tends to 0}
%\frac{\ext_{[\omega_n]_*(x_0)}(F)}{\ext_{x_0}(F)}\to 0
%\end{equation}
%as $n\to \infty$.
%As shown in Walsh \cite[]{Walsh}, Busemann's function for a Teichm\"{u}ller ray $\ray$ issued at $x_0 \in \teich(\Sigma_g)$ directed by a uniquely ergodic projective lamination $[F]$  is expressed as $\displaystyle B_\ray(x)=\frac{1}{2}\log \frac{\ext_x(F)}{\ext_{x_0}(F)}$
%(See also \cite{Miy2} and \cite{MiyUnif}.)
%Liu and Su \cite{LiuSu1} proved that the Gardiner--Masur compactification of $\teich_{g}$ is the horofunction compactification with respect to the Teichm\"uller distance.
%Since $F$ is uniquely ergodic, it is indecomposable and has only one ergodic component.
%Hence, by Walsh's formula for Busemann functions \cite[Theorem 1]{Walsh}, 
%Therefore,  in our situation, we have
%\[
%\lim_{t\to \infty}\left(d_T(\omega_n(x_0),\ray_{F}^{x_0}(t))-t\right)
%=\frac{1}{2}\log\frac{\ext_{\omega_n(x_0)}(F)}{\ext_{x_0}(F)}
%\]
%(see also \cite{Miy2} and \cite{MiyUnif}).
Then, by definition, there exists a sequence of distinct elements $([\omega_n])$ of $G$ such that  $\lim_{t\to \infty}(d_T([\omega_n]_*(x_0),\ray^{F}_{x_0}(t))-t)=-\infty$.
Therefore, there exists a monotone increasing sequence $(t_n)$ tending to $\infty$ such that
\begin{equation}
\label{infinity}
d_T([\omega_n]_*(x_0),\ray^{F}_{x_0}(t_n))-t_n\to -\infty
\end{equation}
as $n\to \infty$.
Set
\[
D_n=d_T([\omega_n]_*(x_0),\ray^{F}_{x_0}(t_n)).
\]
Let $h_n\colon M_0\to M_{t_n}$ be the Teichm\"{u}ller map homotopic to
$f_{t_n}\circ \omega_n\circ f_0^{-1}$.
Then the maximal dilatation of $h_n$ is $e^{2D_n}$, and by (\ref{infinity}), we have  $t_n-D_n\to \infty$ as $n\to \infty$.
Moreover, since $h_n$ was assumed to be homotopic to $f_{t_n}\circ \omega_n\circ f_0^{-1}$, we see that
$f_{t_n}$ is homotopic to $h_n\circ f_0\circ \omega_n^{-1}$.
Thus we have shown that (b) holds.

Conversely, suppose that (b) holds. Since $h_n$ has maximal dilatation $e^{2D_n}$, (\ref{extreme}) gives
\[
\frac{\ext_{\omega_n(x_0)}(F)}{\ext_{x_0}(F)}
\le e^{2D_n}
\frac{\ext_{\ray^{F}_{x_0}(t_n)}(F)}{\ext_{x_0}(F)}
=e^{2D_n-2t_n}\to 0
\]
as $n\to \infty$.
By Walsh's theorem \cite{Walsh}, when $F$ is minimal and uniquely ergodic, $\busemann_F^{x_0}(x)=\frac{1}{2}\log \frac{\ext_x(F)}{\ext_{x_0}(F)}$.
Therefore, we have $\busemann_F^{x_0}(\omega_n(x_0) \to -\infty$, which means that  $[F]$ is a horospherical limit point of $G$.
\end{proof}

\section{Proof of the main theorem}

In this section, we shall prove \Cref{thm:main1}.
Before starting to consider each case, we recall a result of Farb and Masur \cite{FaMa1}.
Following them, we say that a Teichm\"uller geodesic ray $\ray^F_{x_0}$ is  \emph{almost minimising} if there are constants $t_0$, $C>0$ such that 
\[
d_{\moduli}(\varpi(\ray^F_x(t_0)),\varpi(\ray^F_{x_0}(t)))\ge |t-t_0|-C.
\]
In the paper mentioned above, Farb and Masur showed the following.

\begin{theorem}[Theorem 1.4 in Farb--Masur {\cite{FaMa1}}]
    \label{thm:Farb--Masur}
For $[F]\in \pmf$, the following two conditions are equivalent.
\begin{itemize}
\item For any point $x_0 \in \teich(\Sigma_g)$, the  Teichm\"uller geodesic ray $\ray^F_{x_0}$ is almost distance minimising.
\item The quadratic differential $q_{F,x_0}$ is mixed Strebel; that is, the vertical foliation $F$ contains a cylinder foliated by  non-singular closed trajectories.
\end{itemize}
\end{theorem}

Now we start the proof of \Cref{thm:main1}.

\smallskip
\noindent {\bf The if part of (1).}
Suppose that $F$ has a non-singular closed leaf.
Then by definition, $q_{F,x_0}$ is mixed Strebel.
This implies by \Cref{thm:Farb--Masur} that $\ray^F_{x_0}$ is almost minimising.
Therefore, for any $x \in \teich(S)$, there exists a constant $C'$ such that 
\begin{equation}
\label{almost minimising}
d_{\moduli}(\varpi(x), \varpi (\ray^F_{x_0}(t))) \geq t-C'.
\end{equation}
This means that for any sequence $([\omega_n])$ in $\mcg(\Sigma_n)$, we have $d_T([\omega_n]_*(x), \ray^F_{x_0}(t)) \geq t -C'$, and hence $b_F^{x_0}(\omega_n(x))\geq -C'$, and it cannot go to $-\infty$.
By definition, this means that $[F]$ is not a horospherical limit point.

We shall next show that $[F]$ is a big horospherical limit point.
%By \cref{almost minimising}, for any $g\in \mcg(\Sigma_g)$, we have $d_T(gx_0, \ray^F_x(t))\geq t-C$.
%Therefore, $\busemann_F^{x_0}(gx_0) 
Let $\gamma$ be a simple closed curve isotopic to a closed leaf of $F$.
Let $\tau_\gamma \colon \Sigma_g \to \Sigma_g$ be the Dehn twist around $\gamma$, which is regarded as an action on $\teich(\Sigma_g)$.
We note that $\tau_\gamma(F)=F$ since $\gamma$ is isotopic to a leaf of $F$.

Let $F=\sum_{j=1}^k G_j$ be the ergodic decomposition of $F$.
(See \cref{mf}.)
In \cite[Theorem 1]{Walsh}) Walsh showed that for any measured foliation $\mathcal G$, the following holds:
\begin{equation}
\label{Walsh equality}
\lim_{t\to \infty}e^{-t}\ext_{\ray^F_{x_0}(t)}(\mathcal G)^{1/2}=
\left(
\sum_{j=1}^k\frac{I(G_j,\mathcal G)^2}{I(G_j,H(q_{F,x_0}))}
\right)^{1/2},
\end{equation}
where $H(q_{F,x_0})$ denotes the horizontal measured foliation of the quadratic differential $q_{F, x_0}$.
We denote the right hand side of (\ref{Walsh equality}) by $\mathcal{E}_{F, x_0}(\mathcal G)$.

Since $\tau_\gamma$ fixes $F$, by the uniqueness of  the ergodic decomposition, $\tau_\gamma$ just permutes $\{G_j\}$.
Indeed, $\tau_\gamma$ fixes each $G_j$, but the following argument holds for any mapping class fixing $F$ and hence permuting $\{G_j\}$.
Therefore we have the following.
\begin{equation}
\label{Walsh}
\begin{split}
&\sup_{\mathcal G \in \mf}
\frac{\mathcal{E}_{F,x_0}(\mathcal G)}{\ext_{[\tau_\gamma^n]_*(x)}(\mathcal G)^{1/2}} \\
&=\sup_{\mathcal G \in \mf}\left(
\sum_{j=1}^k\frac{1}{I(G_j,H(q_{F,x_0}))}\frac{I(\tau_\gamma^{-n}(G_j),\tau^{-n}_\gamma(\mathcal G))^2}{\ext_{x}(\tau_\gamma^{-n}(\mathcal G))}\right)^{1/2}\\
&=\sup_{\mathcal G \in \mf}\left(
\sum_{j=1}^k\frac{1}{I(G_j,H(q_{F,x_0}))}\frac{I(G_j,\tau_\gamma^{-n}(\mathcal G))^2}{\ext_{x}(\tau_\gamma^{-n}(\mathcal G))}\right)^{1/2}\\
&=\sup_{\mathcal G \in \mf}\frac{\mathcal{E}_{F,x_0}(\mathcal G)}{\ext_{x}(\mathcal G)^{1/2}} 
\end{split}
\end{equation}
By Kerckhoff's formula,  
\begin{align*}
\busemann_F^{x_0}(x)&=\log \left(\lim_{t\to \infty}e^{-t}\sup_{\mathcal G \in \mf }\left(\frac{\ext_{\ray^F_{x_0}(t)}(\mathcal G)}{\ext_{x}(\mathcal G)}\right)^{1/2}\right)\\
&=\log\left(\sup_{\mathcal G \in \mf}\frac{\mathcal{E}_{F,x_0}(\mathcal G)}{\ext_{x}(\mathcal G)^{1/2}} \right).
\end{align*}
Therefore, \Cref{Walsh} implies that $\busemann_F^{x_0}(x)=\busemann_F^{x_0}([\tau_\gamma^n]_*(x))$, and that $[F]$ is a big horospherical limit point.

%Recall that $d_{\moduli}(\varpi(x),\varpi(y))= \inf_{g \in \mcg(\Sigma_g)}d_T(gx,y)$.

%We note that $\tau_\gamma(F)=F$ since $\gamma$ is a leaf of $F$.
%We have, by definition,$$b_F^x(\tau^n(x_0))=\lim_{t\to \infty}(d_T(\tau^n(x_0), \ray^F_x(t))-t=\lim_{t\to \infty}(d_T(x_0, \ray^F_{\tau^{-n}(x)}(t))-t.$$

\smallskip
\noindent{\bf The proof of (2)}
Suppose that $F$ is either non-minimal or non-uniquely ergodic, and does not have a non-singular closed leaf.
Then by \Cref{Masur}, $[F]$ is not a conical limit point, and hence $\ray_{x_0}^F$ is divergent.
On the other hand, by \Cref{thm:Farb--Masur}, $\ray^F_{x_0}$ is not almost minimising.
This means, by definition, that for any $N$, there is $t_N>0$ such that
\[
d_{\moduli}(\varpi(x_0),\varpi(\ray^F_{x_0}(t_N)))<t_N-N
\]
We set $D_n=d_{\moduli}(\varpi(x_0),\varpi(\ray^F_x(t_N)))+1$ and
$\ray^F_{x_0}(t)=(M_t,f_t)$ for $t\ge 0$.
Then by \eqref{eq:distance_modulispace}, for any $N$, there is $[\omega_n]\in \mcg(\Sigma_{g})$ such that
\[
d_T([\omega_n]_*(x_0),\ray^F_{x_0}(t_N))<D_n.
\]
Hence, for each $N$, there is an $e^{2D_N}$-quasi-conformal map $h_N\colon M_0\to M_{t_N}$ such that $h_N\circ f_0\circ \omega_n^{-1}$ is homotopic to $f_{t_N}$.
Furthermore,
since $t_N-D_N>N-1$, we have $t_N-D_N \to \infty$ as $N\to \infty$. 
It follows from \Cref{prop:horospherical_limit_points2}, that $[F]$ is a horospherical limit point of $\mcg(\Sigma_{g})$.

\smallskip
\noindent{\bf The proof of (3).}
This is just a restatement of a part of \Cref{Masur}.

\smallskip
\noindent{\bf The proof of (4).}
Suppose that $F$ is minimal and uniquely ergodic.
Then it cannot contain a non-singular closed leaf.
Therefore, by \Cref{thm:Farb--Masur}, $\ray_{x_0}^F$ is not almost minimising.
Then $\ray_{x_0}^F$ is either recurrent or divergent but not almost minimising.
In the first case, $[F]$ is a conical limit point, by definition.
In the second case, by the same argument as (2), we see that $[F]$ is a non-conical horospherical limit point.

We shall next show that the two possibilities: those of conical limit points and of non-conical horospherical limit points really occur.
If we let $F$ be a stable measured foliation of a pseudo-Anosov map, which is always minimal and uniquely ergodic, then $\ray_{x_0}^F$ is recurrent.
Therefore, the first possibility  really occurs.

In \cite{CheungMa},  Cheung and Masur showed that there is a minimal and uniquely ergodic measured foliation $F\in \mf$ such that $\ray^F_{x_0}$ is divergent.
This means by definition that  $\varpi(\ray^F_{x_0}(t))$ ($t\ge 0$) eventually leaves any compact subset of $\moduli_{g}$, and hence that $[F]$ is not a conical limit point of $\mcg(\Sigma_{g})$.

On the other hand, since $F$ is minimal and uniquely ergodic, it does not have a non-singular closed leaf.
Therefore, by \Cref{thm:Farb--Masur}, $\ray_{x_0}^F$ is not almost minimising, and by the same argument as (2), we see that $[F]$ is a horospherical limit point.

\smallskip
\noindent{\bf The only if part of (1).}
Suppose that  $[\lambda]$ is a big horospherical limit point, but not a horospherical limit point.
Then by (2), $\lambda$ must either be minimal and uniquely ergodic or contain a non-singular closed leaf.
By (4), $\lambda$ cannot be minimal and uniquely ergodic.
Therefore, $\lambda$ must contain a non-singular closed leaf.

%By \cref{thm:Farb--Masur}, the ray cannot be almost distance minimising.
%Hence, for any $N$, there is $t_N>0$ such that
%\[
%d_{\moduli}(\varpi(x_0),\varpi(\ray^F_{x_0}(t_N)))<t_N-N
%\]
%We set $D_n=d_{\moduli}(\varpi(x_0),\varpi(\ray^F_x(t_N)))+1$ and
%$\ray^F_{x_0}(t)=(M_t,f_t)$ for $t\ge 0$.
%Then by \eqref{eq:distance_modulispace}, for any $N$, there is $[\omega_n]\in \mcg(\Sigma_{g})$ such that
%\[
%d_T([\omega_n]_*(x_0),\ray^F_{x_0}(t_N))<D_n.
%\]
%Hence, for each $N$, there is $e^{2D_N}$-quasi-conformal map $h_N\colon M_0\to M_{t_N}$ such that $h_N\circ f_0\circ \omega_n^{-1}$ is homotopic to $f_{t_N}$.
%Furthermore,
%since $t_N-D_N>N-1$, $t_N-D_N \to \infty$ as $N\to \infty$. 
%It follows from \Cref{prop:horospherical_limit_points2}, that $[F]$ is a horospherical limit point of $\mcg(\Sigma_{g})$.

\smallskip
\noindent {\bf The proof of (5).}
It remains to show that the complement of $\Lambda_C(\mcg(\Sigma_g))$ is a null set in terms of the Thurston measure. Since the Teichm\"uller geodesic flow is conservative, almost every Teichm\"uller ray is recurrent in the moduli space $\mathcal{M}_g$ (see \cite{MasurInterval} and \cite{Veech}). This means that almost every point in $\pmf$ is a conical limit point in terms of the Thurston measure, which means that $\Lambda_C(\mcg(\Sigma_g))$ is a full set, and its complement is a null set.

%\subsection{Non-horospherical limit points}
%From Farb--Masur \cite{FaMa1}, we have the following characterization of non-horospherical limit points.
%
%\begin{theorem}
%Let $[\lambda]\in \pmf$ is not a horospherical limit point if and only if $[\lambda]$ contains a foliated cylinder.
%\end{theorem}


\begin{thebibliography}{10}

\bibitem{BeMa}
{\sc Beardon, A.~F., and Maskit, B.}
\newblock {L}imit points of {K}leinian groups and finite sided fundamental
  polyhedra.
\newblock {\em Acta Math. 132\/} (1974), 1--12.

\bibitem{Bi}
{\sc Bishop, C.~J.}
\newblock On a theorem of {B}eardon and {M}askit.
\newblock {\em Ann. Acad. Sci. Fenn. Math. 21}, 2 (1996), 383--388.

\bibitem{Bi2}
{\sc Bishop, C.~J.}
\newblock The linear escape limit set.
\newblock {\em Proc. Amer. Math. Soc. 132}, 5 (2004), 1385--1388.

\bibitem{CheungMa}
{\sc Cheung, Y., and Masur, H.}
\newblock {A} divergent {T}eichm{\"u}ller geodesic with uniquely ergodic
  vertical foliation.
\newblock {\em Israel J. Math. 152\/} (2006), 1--15.

\bibitem{FaMa1}
{\sc Farb, B., and Masur, H.}
\newblock {T}eichm{\"u}ller geometry of moduli space, {I}: distance minimizing
  rays and the {D}eligne-{M}umford compactification.
\newblock {\em J. Differential Geom. 85}, 2 (2010), 187--227.

\bibitem{FLP}
{\sc Fathi, A.;~Po{\'e}naru, V., and Laudenbach, F.}
\newblock {\em Travaux de {T}hurston sur les surfaces}, vol.~66 of {\em
  Ast\'{e}risque}.
\newblock Soci\'{e}t\'{e} Math\'{e}matique de France, Paris, 1979.
\newblock S\'{e}minaire Orsay, With an English summary.

\bibitem{Gard1}
{\sc Gardiner, F.~P.}
\newblock {M}easured foliations and the minimal norm property for quadratic
  differentials.
\newblock {\em Acta Math. 152}, 1-2 (1984), 57--76.

\bibitem{GaMa}
{\sc Gardiner, F.~P., and Masur, H.}
\newblock {E}xtremal length geometry of {T}eichm{\"u}ller space.
\newblock {\em Complex Variables Theory Appl. 16}, 2-3 (1991), 209--237.

\bibitem{HM}
{\sc Hubbard, J., and Masur, H.}
\newblock {Q}uadratic differentials and foliations.
\newblock {\em Acta Math. 142}, 3-4 (1979), 221--274.

\bibitem{Ivanov}
{\sc Ivanov, N.~V.}
\newblock {I}sometries of {T}eichm{\"u}ller spaces from the point of view of
  {M}ostow rigidity.
\newblock In {\em Topology, ergodic theory, real algebraic geometry}, vol.~202
  of {\em Amer. Math. Soc. Transl. Ser. 2}. Amer. Math. Soc., Providence, RI,
  2001, pp.~131--149.

\bibitem{LeiKen}
{\sc Kent, IV, R.~P., and Leininger, C.~J.}
\newblock {S}hadows of mapping class groups: capturing convex cocompactness.
\newblock {\em Geom. Funct. Anal. 18}, 4 (2008), 1270--1325.

\bibitem{Ker}
{\sc Kerckhoff, S.~P.}
\newblock The asymptotic geometry of {T}eichm\"uller space.
\newblock {\em Topology 19}, 1 (1980), 23--41.

\bibitem{LM}
{\sc Lecuire, C., and Mj, M.}
\newblock Horospheres in degenerate 3-manifolds.
\newblock {\em Int. Math. Res. Not. IMRN}, 3 (2018), 816--861.

\bibitem{MasurInterval}
{\sc Masur, H.}
\newblock {I}nterval exchange transformations and measured foliations.
\newblock {\em Ann. of Math. (2) 115}, 1 (1982), 169--200.

\bibitem{Mas2}
{\sc Masur, H.}
\newblock Two boundaries of {T}eichm\"uller space.
\newblock {\em Duke Math. J. 49}, 1 (1982), 183--190.

\bibitem{Mas}
{\sc Masur, H.}
\newblock Hausdorff dimension of the set of nonergodic foliations of a
  quadratic differential.
\newblock {\em Duke Math. J. 66}, 3 (1992), 387--442.

\bibitem{MW}
{\sc Masur, H.~A., and Wolf, M.}
\newblock Teichm\"uller space is not {G}romov hyperbolic.
\newblock {\em Ann. Acad. Sci. Fenn. Ser. A I Math. 20}, 2 (1995), 259--267.

\bibitem{McPa}
{\sc McCarthy, J., and Papadopoulos, A.}
\newblock {D}ynamics on {T}hurston\textquotesingle s sphere of projective
  measured foliations.
\newblock {\em Comment. Math. Helv. 64}, 1 (1989), 133--166.

\bibitem{Min}
{\sc Minsky, Y.~N.}
\newblock Extremal length estimates and product regions in {T}eichm\"uller
  space.
\newblock {\em Duke Math. J. 83}, 2 (1996), 249--286.

\bibitem{Po}
{\sc Poincar{\'e}, H.}
\newblock M{\'e}moire sur les groupes klein{\'e}ens.
\newblock {\em Acta Math. 3\/} (1883), 49--92.

\bibitem{ST}
{\sc Su, W., and Tan, D.}
\newblock {Horospheres in Teichm{\"u}ller space and mapping class group}.
\newblock {\em Annales de l'Institut Fourier 73}, 4 (2023), 1677--1707.

\bibitem{Su}
{\sc Sullivan, D.}
\newblock On the ergodic theory at infinity of an arbitrary discrete group of
  hyperbolic motions.
\newblock In {\em Riemann surfaces and related topics: {P}roceedings of the
  1978 {S}tony {B}rook {C}onference ({S}tate {U}niv. {N}ew {Y}ork, {S}tony
  {B}rook, {N}.{Y}., 1978)\/} (1981), vol.~No. 97 of {\em Ann. of Math. Stud.},
  Princeton Univ. Press, Princeton, NJ, pp.~465--496.

\bibitem{Tu}
{\sc Tukia, P.}
\newblock Conservative action and the horospheric limit set.
\newblock {\em Ann. Acad. Sci. Fenn. Math. 22}, 2 (1997), 387--394.

\bibitem{Veech}
{\sc Veech, W.~A.}
\newblock Gauss measures for transformations on the space of interval exchange
  maps.
\newblock {\em Ann. of Math. (2) 115}, 1 (1982), 201--242.

\bibitem{Walsh}
{\sc Walsh, C.}
\newblock {T}he asymptotic geometry of the {T}eichm{\"u}ller metric.
\newblock {\em Geom. Dedicata 200\/} (2019), 115--152.

\end{thebibliography}
\end{document}